\documentclass[11pt]{amsart}

\usepackage[utf8]{inputenc}
\usepackage[english]{babel}
 
\usepackage{anysize}
\usepackage{amsfonts}
\usepackage{amssymb}
\usepackage{amsmath}
\usepackage{amsthm}
\usepackage[alphabetic,backrefs,lite]{amsrefs} 
\usepackage[all]{xy}
\usepackage{enumerate}
\usepackage{multirow}
\usepackage{geometry}
\usepackage{color}
\usepackage{colonequals}
\usepackage{mathrsfs}

\definecolor{darkgreen}{rgb}{0,0.5,0}

\usepackage[
        draft=false,
        colorlinks, citecolor=darkgreen,
        backref,       
        pdfauthor={Filippo F. Favale and Sonia Brivio},
        pdftitle={Coherent systems and BGN extensions on nodal reducible curves},
        linktocpage]{hyperref}

\newtheorem{theorem}{Theorem}[section]
\newtheorem{proposition}[theorem]{Proposition}
\newtheorem{definition}[theorem]{Definition}

\newtheorem{corollary}[theorem]{Corollary}
\newtheorem{lemma}[theorem]{Lemma}

\newtheorem{remark}[theorem]{Remark}

\DeclareMathOperator{\GL}{GL}
\DeclareMathOperator{\Ker}{Ker}
\DeclareMathOperator{\coker}{coker}
\DeclareMathOperator{\Supp}{Supp}

\DeclareMathOperator{\Ext}{Ext}
\DeclareMathOperator{\Hom}{Hom}

\DeclareMathOperator{\Tor}{Tor}

\DeclareMathOperator{\rank}{rank}
\DeclareMathOperator{\Rk}{rk}

\newcommand{\N}{\mathbb{N}}
\newcommand{\Z}{\mathbb{Z}}
\newcommand{\Q}{\mathbb{Q}}

\newcommand{\R}{\mathbb{R}}

\newcommand{\OO}{\mathcal{O}}

\newcommand{\U}{\mathcal{U}}

\newcommand{\id}{id}

\newcommand{\ev}{ev}

\newcommand{\w}{{\underline w}}
\newcommand{\wa}{({\underline w},\alpha)}
\DeclareMathOperator{\mrank}{\underline{rk}}
\DeclareMathOperator{\wdeg}{deg_{\underline{w}}}
\DeclareMathOperator{\wrank}{rk_{\underline{w}}}
\DeclareMathOperator{\Gr}{Gr}
\DeclareMathOperator{\Ima}{Im}
\newcommand{\Gwat}{\tilde{\mathcal{G}}_{(C,\w),\alpha}}
\newcommand{\Gwa}{\mathcal{G}_{(C,\w),\alpha}}
\newcommand{\GwL}{\mathcal{G}_{(C,\w),L}}

\newcommand{\cW}{\mathcal{W}}

\newcommand{\cV}{\mathcal{V}}

\newcommand{\cO}{\mathcal{O}}

\newcommand{\bC}{\mathbb{C}}
\newcommand{\cU}{\mathcal{U}}

\newcommand{\cHom}{\mathcal{H}om}
\newcommand{\cExt}{\mathcal{E}xt}

\numberwithin{equation}{section}

\begin{document}

\title[Coherent systems and BGN extensions]{Coherent systems and BGN extensions on nodal reducible curves}

\author{Sonia Brivio}
\address{Dipartimento di Matematica e Applicazioni,
	Universit\`a degli Studi di Milano-Bicocca,
	Via Roberto Cozzi, 55,
	I-20125 Milano, Italy}
\email{sonia.brivio@unimib.it}

\author{Filippo F. Favale}
\address{Dipartimento di Matematica e Applicazioni,
	Universit\`a degli Studi di Milano-Bicocca,
	Via Roberto Cozzi, 55,
	I-20125 Milano, Italy}
\email{filippo.favale@unimib.it}

\date{\today}
\thanks{
Both authors are partially supported by INdAM - GNSAGA.\\
}
\maketitle

\begin{abstract}
Let $(C,\w)$ be a polarized nodal reducible curve. In this paper we consider coherent systems of type $(r,d,k)$ on $C$ with $k < r$. We prove that the moduli spaces of $\wa$-stable coherent systems stabilize for large $\alpha$ and we generalize several results known for the irreducible case when we chose a good polarization. Then, we study in details the components of moduli spaces containing coherent systems arising from locally free sheaves.
\end{abstract}

\section*{Introduction}

In this paper we deal with coherent systems on nodal reducible curves. 
Coherent systems, on a smooth curve, are pairs $(E,V)$ where $E$ is a locally free sheaf and $V$ is a subspace of global sections of $E$. Hence they can be seen as a generalisation of linear systems and they are closely related to higher rank Brill-Noether theory. They were introduced for a smooth curve in \cite{BD91,Ber,LP} under different names and have been studied extensively by several authors (see \cite{BGMN, BGN97, New11} for relevant results and \cite{Bra} for a survey). 
For any real parameter $\alpha$, the notion of $\alpha$-stability has been introduced and it gives a family of coarse moduli spaces parametrizing coherent systems  $(E,V)$ of  type $(r,d,k)$, i.e. with $\Rk(E)=r$, $\deg(E)=d$ and $\dim V= k$.  For comparison of different notions of stability in moduli theory see for instance  \cite{BB12}. 
\vspace{2mm}

Actually, for any type, there are only a finite number of distinct moduli spaces and the ``terminal'' one (which corresponds to the biggest possible choice for $\alpha$ for which the moduli space is not empty) can be described by using different approaches according to whether $k \geq  r$ or $k < r$.  In the first case, the description is related to Quot scheme of quotients of the trivial bundle of rank $k$ (see \cite{BGMN}) while in the second case the moduli space parametrizes extensions of a sheaf by the trivial bundle of rank $k$ (see \cite{BGN97}). These extensions are said BGN extensions.
\vspace{2mm}

Coherent systems can be defined even on a singular reduced curve but in this case, in order to have compact moduli spaces, one needs to consider not only locally free sheaves but also torsion free sheaves. In particular, when $C$ is a nodal curve, a coherent system is a pair $(E,V)$, where $E$ is a depth one sheaf and $V\subseteq H^0(E)$. A notion of $\w$-rank and $\w$-degree can be defined on nodal curves once we fix a polarization $\w$ on $C$. Hence, for any coherent system  one can  define the type  and the notion of $\wa$-stability, as well the $\w$-stability for  any depth one sheaf.
In \cite{KN}, for any $\alpha \in {\mathbb R}$, it has been proved the existence of coarse moduli spaces $\Gwa(r,d,k)$ parametrizing families of $\wa$-stable coherent systems of type $(r,d,k)$. For details, the reader can see Section \ref{SEC:Preliminaries}.
\vspace{2mm}

Many results which hold in the smooth case (concerning, for example, dimension, irreducibility, values of the parameters for which these moduli spaces are empty and smoothness),  have been extended to the case of irreducible nodal curves in \cite{Bho,Bho07}. The authors have begun studying coherent systems on reducible nodal curves of compact type in \cite{BFCoh},  considering coherent systems $(E,V)$ with $ k \geq r$ and assuming  that the rank of the restriction of $E$ to each component of the curve is the same. In this paper we consider also curves which are not of compact type and we deal with the case $k<r$.
\hfill\par
Let $(C,\w)$ be a polarized nodal curve with smooth irreducible components of genus at least $2$. As in the case of smooth or irreducible nodal curves, we prove that for any type $(r,d,k)$ with $k < r$ there are only finitely many distinct moduli spaces (see Lemma \ref{LEM:pcritici} and Proposition \ref{PROP:alphaL}) so there is a ``terminal'' moduli space which will denote by $\GwL(r,d,k)$. In order to describe this moduli space we generalize the notion of BGN extensions to nodal reducible curves (see Section \ref{SEC:BGN}). Proceeding as in the smooth case, more technicalities are involved as the curve is reducible, depending also on the chosen polarization and the behaviour of these moduli spaces seems very wild. 
The situation became a little bit better if one choses a polarization which is {\it good}. This class of polarizations was introduced by the authors in \cite{BFPol} by observing that depth one sheaves on nodal curves equipped with good polarizations reflect a lot of properties that hold for vector bundles on smooth curves. For curves of compact type, good polarizations are exactly those for which $\cO_C$ is $\w$-stable (in general, $\cO_C$ is not even $\w$-semistable!). The definition of good polarization on a nodal curve is rather technical so we refer to the preliminaries in Section \ref{SEC:Preliminaries}.
\vspace{2mm}

In Section \ref{SEC:3}, by fixing a good polarization $\w$ on a nodal curve $C$  we are able to  describe (in analogy to the smooth case) coherent systems of the moduli space $\GwL(r,d,k)$ as BGN extensions of  depth one sheaves  which  are  $\w$-semistable (see Theorem \ref{THM:BGNext} and Theorem \ref{THM:cs}).
\vspace{2mm}

In the last section (i.e. Section \ref{SEC:4}) we analyze in details the moduli space
$\GwL(r\cdot \underline{1},d,k)$ parametrizing coherent systems 
$(E,V)$ where the restriction of $E$ to any component has rank $r$. We are interested in these moduli spaces as they contains coherent systems arising from locally free sheaves. 
For any good polarization $\w$ on a nodal curve,
we give a necessary and sufficient condition for emptyness of $\GwL(r\cdot \underline{1},d,k)$. Moreover, we give a description of  all   irreducible components  containing  coherent systems $(E,V)$ with $E$ locally free, generalizing the picture of the situation in the case of smooth curves. 
Such components are birational to Grassmannian fibrations over irreducible components of the moduli space of $\w$-semistable depth one sheaves whose $\w$-rank and $\w$-degree are $(r-k)$ and $d$ respectively (see Theorem \ref{THM:main}). 
Finally, we investigate how coherent systems restrict to  irreducible components of the curve $C$ (see Corollary \ref{COR:restriction}).

\section{Technical results and preliminaries}
\label{SEC:Preliminaries}

In this section we recall some definitions and  we state relevant technical results about nodal curves,  depth one sheaves  on them and coherent systems.

\subsection{Nodal curves}
\label{SUBSEC:nodalcurves}
Let $C$ be a connected reduced nodal curve over the complex field (i.e. having only ordinary double points as singularities).  We will denote by $\gamma$ the number of  irreducible components and $\delta$ the number of nodes  of $C$. We will assume that each irreducible component $C_i$ is a smooth curve of genus $g_i \geq 2$. For the  theory of nodal curves see \cite[Ch X]{ACG}. 
We will denote by 
$$ \nu \colon  C^{\nu}= {\bigsqcup}_{i=1}^{\gamma}C_i \to C$$
the normalization map. 
From the exact sequence:
$$ 0  \to \OO_C \to \nu_*\nu^*(\OO_C) \to \bigoplus_{p\in Sing(C)}{\mathbb C_p} \to 0, $$
we deduce that 
$\chi(\OO_C) = \sum_{i=1}^{\gamma}\chi(\OO_{C_i}) -\delta$,
and we obtain the {\it arithmetic  genus}  of $C$: 
\begin{equation}
\label{EQ:paC}
p_a(C)=\sum_{i=1}^{\gamma}g_i+\delta-\gamma+1.
\end{equation}

\noindent For any irreducible component $C_i$ let   $C_i^c$  be the closure of $C\setminus C_i$. We set $\Delta_i=C_i\cdot C_i^c$, and we denote by  $\delta_i$ its  degree, i.e. the number of nodes of $C$ on $C_i$. We recall that $C$ can be  embedded in a smooth projective surface $S$ (see \cite{A79}) and this yields the exact sequence
\begin{equation}
\label{EXSEQ:CB}
0\to \OO_{C_i^c}(-\Delta_i) \to \OO_C \to \OO_{C^i} \to 0.
\end{equation}
For more details on this sequence see \cite{BFVec}. 
Since $C$ is  a local complete intersection,  Serre duality Theorem holds. There exists on $C$ a dualizing sheaf $\omega_C$, it is an invertible sheaf as $C$ is Gorestein.  We recall that, for any $i= 1 \dots \gamma$, we have  $\omega_C|_{C_i}=\omega_{C_i}(\Delta_i)$.
\hfill\par 
\hfill\par
Finally, we recall some technical  results.
Let $p$ be a node and denote by $C_{i_1}$ and $C_{i_2}$ the two components such that $p\in C_{i_1}\cap C_{i_2}$. Following the notations of \cite{Ses}, chap. 8,  we set:
$$\OO_{x_{i_k}}= \OO_{C_{i_k},p}, \quad m_{x_{i_k}} = m_{C_{i_k},p}, \quad \OO_p= \OO_{C,p} \quad m_{p} = m_{C,p}.$$
Then:
$$\OO_p = \{ (f,g) \in \OO_{x_{i_1}} \oplus \OO_{x_{i_2}} \,\vert\,  f(p) = g(p) \}, \quad m_p = m_{x_{i_1}}
\oplus m_{x_{i_2}}.$$
The  
 isomorphisms $\OO_{x_{i_k}} \simeq m_{x_{i_k}}$ 
obtained by sending $f \mapsto ft_{i_k}$, where $t_{i_k}$ is a local coordinate on $C_{i_k}$ at $p$, induce an isomorphism 
$\OO_{x_{i_1}} \oplus \OO_{x_{i_2}} \simeq m_p$. 
We have the  following exact sequences of $\OO_p$-moduli:
\begin{equation}
\label{eq1}
0 \to \OO_p \to \OO_{x_{i_1}} \oplus \OO_{x_{i_2}} \to {\mathbb C} \to 0\qquad \mbox{ and }\qquad 0 \to m_p \to \OO_p \to {\mathbb C} \to 0.
\end{equation}

\begin{lemma} 
\label{homloc}
In the above hypothesis we have:
\hfill\par
\begin{enumerate} [(a)]
\item{} 
$\Hom_{\OO_p}(\OO_{x_{i_k}},\OO_{x_{i_j}})=\begin{cases} 0 & \mbox{if } k\neq j\\
\OO_{x_{i_j}} & \mbox{if } k=j
\end{cases}$, $ \Hom_{\OO_p}(\OO_{x_{i_k}},\OO_p) \simeq t_{i_k}\cdot  \OO_{x_{i_k}} \simeq m_{x_{i_k}}$;
\item{} $\Hom_{\OO_p}({\mathbb C},\OO_p)=\Hom_{\OO_p}({\mathbb C},m_p)=\Hom_{\OO_p}({\mathbb C},\OO_{x_{i_k}})=0$;
\item{} $\Ext_{\OO_p}^1(\bC, \OO_p)\simeq \bC$,
$\Ext_{\OO_p}^1(m_P,\cO_p) = 
\Ext_{\OO_p}^1(\OO_{x_{i_k}},\cO_p) =0$.
\end{enumerate}
\end{lemma}



\begin{proof}
For $(a)$ see \cite[Lemma 6, pag. 171]{Ses}. $(b)$ and $(c)$ are obtained by applying $\Hom_{\cO_p}(-,\cO_p)$ to the exact sequence \ref{eq1}.
\end{proof}

\begin{definition}
\label{DEF:polarization}
Let $C$ be a nodal curve with $\gamma$  irreducible components. 
A {\bf polarization} on  $C$ is a vector  ${\underline w}= (w_1,\dots,w_{\gamma}) \in {\mathbb Q}^{\gamma}$ such that
\begin{equation}
0 < w_i < 1 \quad \sum_{i=1}^{\gamma}w_i = 1.
\end{equation}
We will say that the pair $(C,\w)$ is a 
{\bf polarized curve}. 
\end{definition}

\subsection{Depth one sheaves on nodal curves}
\label{Subsec:depthone}
We recall the notion of depth one sheaves on nodal curves (see \cite{Ses} for details).

A coherent sheaf $E$ on a reduced curve  is said to be of {\it depth one} if  for any $x \in \Supp(E)$ the stalk $E_x$ is an $\OO_x$-module of depth one. For a nodal curve this is equivalent to say that $E$ is {\it pure of dimension $1$}, i.e.   $\dim F = \dim \Supp(F) = 1$ for every subsheaf $F$ of $E$.  Let $C$ be a nodal curve with smooth irreducible components $C_i$.  Using the notations introduced above, a coherent sheaf $E$ on $C$ is of depth one if the stalk of $E$ at the node $p \in C_{i_1}\cap C_{i_2}$ is isomorphic to $\OO_p^s \oplus \OO_{x_{i_1}}^{a_1} \oplus \OO_{x_{i_2}}^{a_2}$. In particular, any  vector bundle  on $C$ is a sheaf of depth one and any subsheaf of a  depth one  sheaf is of depth one too.
\hfill\par

Let $E$  be a depth one sheaf on  $C$. Its dual sheaf    $E^* =  \cHom_{\OO_C}(E,\OO_C)$  is  of depth one too  and $E$ is reflexive, i.e. $\cHom_{\OO_C}(E^*,\OO_C) \simeq E$.   
In particular, Serre duality   yields   for any $q \geq 0$   the isomorphism  $H^q(E)^* \simeq H^{1-q}(E^* \otimes {\omega}_C)$.


\begin{lemma}
\label{LEM:EXT1}
Let $E$ be a depth one sheaf on $C$,  then $\cExt^1(E,\cO_C)=0$.
\end{lemma}

\begin{proof}
 If $p \in C$ is  a smooth point,  then  $E_p \simeq \OO_p^r$ so  we have $\Ext^1(E_p,\cO_p) = 0$. If $p  \in C_{i_1} \cap C_{i_2}$,  then $ E_p \simeq \OO_{p}^{s} \oplus \OO_{x_{i_1}}^{a_1} \oplus \OO_{x_{i_2}}^{a_2}$. 
By Lemma \ref{homloc}(c) we have $\Ext_{\OO_p}^1(\OO_{x_{i_k}}, \OO_p) = \Ext_{\OO_p}^1(\OO_p, \OO_p) =0$ so $\Ext_{\OO_p}^1(E_p,\OO_p) = 0$.
\end{proof}

Let $(C,\w)$  be a  polarized nodal curve with $\gamma$ smooth  irreducible components $C_i$. 
Let $E$ be a sheaf of depth one on $C$,   we define the  {\it restriction  of $E$ modulo torsion}  on the component $C_i$ as
$ E_i = E \otimes \OO_{C_i} /Torsion$.  We set $r_i= \rank(E_i)$. 
We  define the {\it multirank} of $E$, the ${\underline w}$-{\it rank} of $E$ and the ${\underline w}$-{\it degree} of $E$ respectively as 
\begin{equation}
\label{multirank}
\mrank(E) = (r_1,\dots,r_\gamma)\qquad \wrank (E) = \sum_{i = 1}^{\gamma}w_i r_i\qquad \wdeg E = \chi(E)-\wrank(E) \chi(\OO_C).
\end{equation}
 
For brevity, we will denote by $\underline{1}$ the vector $(1,\dots,1)\in \Z^{\gamma}$. 
Note that ${\underline w}$-rank and ${\underline w}$-degree are not  necessary integers. When $E$ is a vector bundle on $C$, i.e. it is locally isomorphic to $\OO_C^r$, then $\rank(E_i) = r$ $\forall i$, the ${\underline w}$-rank of $E$ is actually $r$ and its multirank is $r\cdot \underline{1}$. Moreover the ${\underline w}$-degree of $E$ is the sum of the degrees of $E_i$ and, in particular, it is an integer too. In general, this is not the case. In \cite {BFPol} the authors have introduced and studied the function $\Delta_{\w}(E)=\wdeg(E)-\sum_{i=1}^{\gamma} \deg(E_i)$.

\begin{lemma}
\label{LEM:wdegree}
Let $(C,\w)$ be a polarized nodal curve.  Let $E$ be a depth one sheaf on $C$,  let $E_i$ its restriction modulo torsion to $C_i$ with $\rank(E_i) = r_i$ and $\deg(E_i) = d_i$. Then the following hold:
\begin{enumerate}[(a)] \item$\Delta_{\w}(E)=\sum_{i=1}^{\gamma}  r_i (1-g_i -w_i \chi(\OO_C))  - \sum_{j=1}^{\delta}s_{p_j}$, where
$s_{p_j}$ is the rank of the free part of the stalk  of $E$ at $p_j$;
\item  $\Delta_{\w}(E\otimes L)=\Delta_{\w}(E)$ for any $L$ line bundle on $C$;
\item if $r_i=r$ for any $i$, then $\Delta_{\w}(E) \geq 0$ and equality holds if and only  $E$ is locally free;
\item let $w_{m}=\min\{w_i\,:\, i=1,\dots,\gamma\}$ and $r_{M}=\max \{r_i\,:\, i=1,\dots,\gamma\}$, we have
$$\Delta_{\w}(E) \geq -(r_{M}-\wrank(E))(p_a(C)-1)\geq -\left(\frac{1}{w_{m}}-1\right)\wrank(E)(p_a(C)-1).
$$
\end{enumerate}
\end{lemma}
\begin{proof}
Statements $(a),(b)$ and $(c)$ have been  proved in \cite{BFPol}*{Lemma 2.4}. 
From $(a)$ we get
$$\Delta_{\w}(E)=- \sum_{i=1}^{\gamma}r_i(g_i-1)  + \sum_{i=1}^{\gamma}w_ir_i(p_a(C)-1)  -\sum_{j=1}^\delta s_{p_j},$$
since  $s_{p_j} \leq r_{M}$, $r_i \leq r_{M}$  and $g_i \geq 2$, we obtain
$$\Delta_{\w}(E) \geq - r_{M} \sum_{i=1}^{\gamma}(g_i-1)  + \wrank(E)(p_a(C)-1)   -\delta r_{M} \geq (\wrank(E)-r_{M})\left(p_a(C)-1\right).$$
For all $j$ we have $r_j\leq \wrank(E)/{w_{m}}$, so $r_{M}\leq \wrank(E)/{w_{m}}$ and this proves $(d)$.
\end{proof}

We recall the definition of good polarization (see \cite{BFPol}).

\begin{definition}
\label{GOOD}
A polarization $\w$ on $C$ is said {\bf good} if $\Delta_{\w}(E) \geq 0$ for all depth one sheaves $E$ on $C$ and equality holds if and only if $E$ is locally free.
\end{definition}

In \cite{BFPol} the authors proved that good polarizations exist on any stable nodal curve $C$ with $p_a(C) \geq 2$. If $\w$ is good, then $\cO_C$ is $\w$-stable and the converse holds when $C$ is a nodal curve of compact type (see \cite[Theorem 3.9]{BFPol}). It is also conjectured that this should hold for any nodal curve. The above lemma justifies the following definition:
\begin{definition}
\label{DEF:lambdaw}
Let $(C,\w)$  be a polarized nodal curve. We set 
$$\lambda_{\w}=\begin{cases}
0 & \mbox{ if } \w \mbox{ is good}\\
\left(\frac{1}{w_{m}}-1\right)(p_a(C)-1) & \mbox{otherwise}.
\end{cases}$$
\end{definition}

Notice that $\lambda_{\w}\geq0$ for all $\w$ and, by Lemma \ref{LEM:wdegree}, if $E$ is a depth one sheaf on $C$ we have $\Delta_{\w}(E)\geq-\lambda_{\w}\wrank(E)$.
\vspace{2mm}

We recall the notion of $\w$-semistability for depth one sheaves on a polarized nodal curve.  

\begin{definition}
A depth one sheaf $E$ is  said $\w$-(semi)stable if for any proper subsheaf $F$ of $E$ we have $\mu_{\w}(F)< \mu_{\w}(E)$ (resp. $\leq$), where $\mu_{\w}(E)=\wdeg(E)/\wrank(E)$ is said $\w$-slope of $E$.
\end{definition}

For any $d\in \Q$ and $r\in \Q_{>0}$, there exists a coarse moduli space $\mathcal{U}_{(C,\w)}(r,d)$ for   families  of $\w$-semistable depth one sheaves on $C$ with prescribed $\w$-rank and $\w$-degree. It is a projective variety.  Moreover, when we fix $(r_1,\dots,r_{\gamma})\in \N^\gamma$, we obtain the moduli space $\mathcal{U}_{(C,\w)}((r_1,\dots,r_{\gamma}),d)$ of classes of $w$-semistable depth one sheaves with prescribed multirank. When we  will consider the subscheme parametrizing  $\w$-stable classes we will use the notation $\mathcal{U}^s_{(C,\w)}(r,d)$ and $\mathcal{U}^s_{(C,\w)}((r_1,\dots,r_{\gamma}),d)$. For details one can see \cite{Ses}. 
We will denote by $[F]$ the $S$-equivalence class of a $\w$-semistable sheaf in its moduli space. When $F$ is $\w$-stable, $[F]$ is the isomorphism class of $F$. In this case, we will denote it simply by $F$.
\vspace{2mm}

When $\w$ is a good polarization, we generalize to $\w$-semistable depth one sheaves a well known result which holds for semistable vector bundles on a smooth curve.

\begin{lemma}
\label{LEM:h0dual}
Let $(C,\w)$ be a polarized nodal curve with $\w$ good. Assume that $E$ is a $\w$-semistable depth-one sheaf with $\wdeg(E)=d>0$. Then $h^0(E^*)=0$. 
\end{lemma}
\begin{proof}
Assume that $h^0(E^*) \neq 0$. As $E^*=\cHom(E,\OO_C)$, we have
$\Hom(E,\OO_C)\neq0$ so there exists a non zero homomorphism $\varphi \colon E \to \OO_C$.
Denote by $G$ the image of $\varphi$, so we have
$$\xymatrix{
E \ar@{->>}[r]_{\psi}\ar@/^1pc/[rr]^{\varphi} & G \ar@{^{(}->}[r] & \cO_C.
}$$
We will show that this implies $d\leq 0$ which contradicts our  hypothesis. If $G=\OO_C$ or $\psi$ is an isomorphism, we conclude respectively by $\w$-semistability of $E$ and by $\w$-stability of $\cO_C$ (which holds since $\w$ is good). Otherwise, we have that $G$ is a depth one sheaf which is a proper subsheaf of $\OO_C$ and a proper quotient of $E$. Then we have
$$d/\wrank(E)=\mu_{\w}(E)\leq \mu_{\w}(G)<\mu_{\w}(\cO_C)=0$$
by $\w$-semistability of $E$ and $\w$-stability of $\cO_C$ respectively.
\end{proof}

We stress that the above result does not hold  when $\w$ is not good.

\subsection{Coherent systems on nodal curves}
\label{Subsec:cohsystem}

Let $(C, \w)$ be a polarized  nodal curve. We will recall the notion of coherent systems on  the curve $C$ (see \cite{KN} for details).

\begin{definition}
A {\it coherent system}  on the curve $C$ is  given by a pair $(E,V)$, where $E$ is a depth one sheaf on $C$ and $V$ is a subspace of $H^0(E)$. \end{definition} 
A {\it coherent subsystem} $(F,U)$ of $(E,V)$ is a coherent system which consists of a subsheaf $F \subseteq E$ and a subspace  $U\subseteq V \cap H^0(F)$. We say that $(F,U)$  is a {\it proper}  subsystem if $(F,U) \not= (0,0)$ and $(F,U) \not= (E,V)$. 
A coherent system $(E,V)$ is said to be of {\it type} $(r,d,k)$ if $\wrank(E) = r$, $\wdeg E = d$ and $\dim V = k$; 
if the multirank of $E$ is $\mrank(E) =(r_1,\cdots,r_{\gamma})$ then it is said to be of {\it multitype} $((r_1,\cdots,r_{\gamma}),d,k)$.

Let  $(E,V)$ be a coherent system on $C$.  Let  $C_i$ be  a component of $C$, from the exact sequence \eqref{EXSEQ:CB}, we obtain that the restriction  map $E \to E \otimes \OO_{C_i}$ is surjective.
Then, the map $E\to E_i$ is also surjective and induces the map of global sections
$$ \rho_i \colon H^0(E)\to H^0(E_i).$$
We define $V_i$ as the  image of $V$ by the map  $\rho_i$. 
We will call $(E_i,V_i)$ the {\it restriction of $(E,V)$ to the curve $C_i$}. A coherent system $(E,V)$ is called {\it generated} if the evaluation map $\ev_V:V\otimes \OO_C\to E$ is surjective.

\begin{lemma}
\label{LEM:generatedsystem}
Let  $(E,V)$  be a generated  coherent system on $(C,\w)$ of type $(r,d,k)$. Then, 
\begin{enumerate}[(a)]
\item $r\leq k$, $(E_i,V_i)$ is generated and $\deg(E_i)\geq 0$;
\item if either $\w$ is good or $\Rk(E_i)=r$ for all $i=1,\dots,\gamma$, then   $d\geq 0$.
\end{enumerate}
\end{lemma}

\begin{proof}
Since $(E,V)$ is generated, the evaluation map $ev_V \colon V \otimes \OO_C  \to E $ is surjective. This implies  $k \geq r$. By \cite{BFCoh}*{Lemma 3.3}, each  restriction $(E_i,V_i)$ is generated too so $\deg(E_i) \geq 0$. $(b)$ follows easily from Lemma \ref{LEM:wdegree} or the definition of good polarization as $\deg(E_i)\geq 0$ for all $i$.
\end{proof}

\begin{definition} Let $(E,V)$ be a coherent system on the curve $C$. Fix $\alpha \in \R$ and a polarization ${\underline w}$ on $C$. We say that $(E,V)$ is $\wa$-(semi)stable if for any proper coherent subsystem $(F,U)$ we have
$$\mu_{{\underline w},\alpha} (F,U) < \mu_{{\underline w},\alpha} (E,V)  \ \  (resp. \leq)$$
where $\mu_{\wa}$ is the $\wa$-slope, which is defined as 
$$\mu_{{\underline w},\alpha} (E,V) = \frac{\wdeg(E)}{\wrank(E)} + \alpha \frac{k}{\wrank(E)}= \mu_{\w}(E) + \alpha \frac{k}{\wrank(E)}.$$
\end{definition}

Fix $(r,d,k)$ with $r,d\in\mathbb{R}$, $r>0$, $k\in \mathbb{N}$ and $\alpha \in {\mathbb R}$ positive. In \cite{KN} it is proved that there exists a projective scheme 
$\tilde{\mathcal G}_{(C,\w),\alpha}(r,d,k)$ which is a coarse moduli space for families of
$\wa$-semistable coherent systems of type
$(r,d,k)$ on the polarized curve $(C,\w)$. 
Moreover, the open subscheme $\Gwa(r,d,k) \subset \tilde{\mathcal G}_{(C,\w),\alpha}(r,d,k)$,  parametrizing $\wa$-stable pairs,  is a coarse moduli space for $\wa$-stable  coherent systems of type $(r,d,k)$. Finally, when we fix $(r_1,\dots,r_{\gamma}) \in {\mathbb N}^{\gamma}$ we obtain the moduli space  $\Gwa((r_1,\dots,r_{\gamma}),d,k)$ parametrizing $\wa$-stable  coherent systems $(E,V)$ of multitype $((r_1,\dots,r_{\gamma}),d,k)$. 
For brevity, if $(E,V)$ is a $\wa$-stable coherent system, we denote by $(E,V)$ also its isomorphism class in its moduli space.
\hfill\par

In the last section we will focus  on  the moduli spaces $\Gwa(r\cdot\underline{1},d,k)$. In particular, we will consider coherent systems $(E,V)$  where  $E$  is  locally free. For these points,  we can  extend  a local smoothness condition, which holds   when $C$ is  smooth or  nodal irreducible  (see \cite{BGMN} and  \cite{Bho}).
We can define the Brill-Noether number
$$\beta_C(r,d,k)= r^2(p_a(C)-1)+1-k(k-d+r(p_a(C)-1)).$$
For any  $(E,V)\in \Gwa(r\cdot \underline{1},d,k)$  with  $E$ locally free, we can define the Petri map
$$\mu_{(E,V)}:V\otimes H^0(\omega_C\otimes E^*)\to H^0(\omega_C\otimes E\otimes E^*),$$
which is given by multiplication by global sections.

\begin{proposition}
\label{PROP:Petri}
Let $(E,V)\in \Gwa(r\cdot \underline{1},d,k)$, assume that $E$ is locally free and that the evaluation map $\ev_V$  is injective. Then the Petri map $\mu_{(E,V)}$ is injective too and the moduli space
$\Gwa(r\cdot \underline{1},d,k)$
is a smooth at $(E,V)$ with dimension $\beta_C(r,d,k)$.
\end{proposition}

\begin{proof}
The result follows from \cite[Prop 2.7]{BFCoh} once we  prove that the Petri map $\mu_{(E,V)}$ is injective. As $\ev_V$ is injective and $E^*\otimes \omega_C$ is locally free, we have an exact sequence
$$0\to V\otimes E^*\otimes \omega_C\to E\otimes E^*\otimes \omega_C.$$
The induced map in cohomology is  injective and it is $\mu_{(E,V)}$. 
\end{proof}

\section{BGN extensions on nodal  reducible curves}
\label{SEC:BGN}

Let $(C, \w)$ be a polarized   nodal curve with $\gamma$ irreducible  smooth components $C_i$ of genus  $g_i \geq 2$ and $\delta$ nodes. 
In this section we generalize the notion of $BGN$ extension defined in \cite{BGN97} to nodal reducible curves. 
\begin{definition}
\label{DEF:BGN}
Let $r$ and $d$  two positive rational numbers and $k$ an integer  such that $0<k<r$. A BGN extension on $C$ of type $(r,d,k)$ is an extension 
$$ 0  \to V \otimes \OO_C \to E \to F \to 0$$
such that 
\begin{itemize}
\item $V$ is a vector space of dimension $k$;
\item $F$ is a depth one sheaf with  $\wrank(F)=r-k$,  $\wdeg(F) = d$;
\item let $\underline{e}= ({\underline e}_1, \dots, {\underline e}_k) \in \Ext^1(F, V \otimes \OO_C)=V\otimes \Ext^1(F,\OO_C)\simeq \Ext^1(F,\OO_C)^{\oplus k}$ be the corresponding extension class,
 $\{{\underline e}_1, \dots, {\underline e}_k\}$ are linearly independent in $\Ext^1(F,\OO_C)$. 
\end{itemize}
\end{definition}
 
Two BGN extensions are said to be {\it equivalent} if they are equivalent as extensions. A BGN extension does not depend on the choice of $F$ in its isomorphism class. 
\vspace{2mm}

In the  sequel, we will be interested in BGN extensions arising from depth one sheaves $F$ which are $\w$-semistable. Note that, when  $C$ is smooth, and $F$ is semistable with $\deg(F)>0$, then  $h^0(F^*)=0$.  So this condition  was required in the definition of BGN extensions given in \cite{BGN97}. This  property does not always occur if $C$ is reducible. Nevertheless, by Lemma \ref{LEM:h0dual}, this holds when $\w$ is a good polarization.

\begin{proposition}
\label{PROP:dualextension}
Consider a BGN extension
$$\underline{e}:\qquad 0 \to V \otimes \OO_C \to E \to F \to 0$$
of a depth one sheaf $F$ and denote by $(r,d,k)$ its type. Then we have:
\begin{enumerate}[(a)]
\item  $E$ is a depth one sheaf on $C$ with $\wrank(E) = r $ and $\wdeg(E) = d$. 
Moreover,  $E$ is locally free  if and only if $F$ is locally free.
\item The following sequence is also exact
$$\underline{e}^*:\qquad  0 \to F^* \to E^* \to V^* \otimes \OO_C \to 0.$$
\item All  BGN extensions on $C$ of $F$ of type $(r,d,k)$ are classified by   $\Gr(k,H^1(F^*)) \simeq \Gr(k, H^0(F\otimes \omega_C)^*)$. 
\end{enumerate}
\end{proposition}

\begin{proof}
(a) It is enough to describe the stalk of $E$ at nodes. If $p \in C_{i_1} \cap C_{i_2}$, the stalk $F_p$ can be written as $\OO_{p}^{s} \oplus \OO_{x_{i_1}}^{a_1} \oplus \OO_{x_{i_2}}^{a_2}$, with $s+a_j= r_{i_j}$. We have the exact sequence of $\OO_p$-moduli 
\begin{equation}
\label{EQ:STALKS}
0 \to \OO_{p}^{k}  \to E_{p}  \to F_{p} \to 0.
\end{equation}
By Lemma \ref{LEM:EXT1} we have $\Ext^1(F_p,\cO_p)^{\otimes k}=0$ so
$E_p \simeq F_p \oplus \OO_p^k \simeq \OO_{p}^{s+k} \oplus \OO_{x_{i_1}}^{a_1} \oplus \OO_{x_{i_2}}^{a_2},$ which proves (a). 
\vspace{2mm}

(b) This follows from Lemma \ref{LEM:EXT1} by applying $\cHom_{\OO_C}(-,\OO_C)$ to the exact sequence defining the BGN extension.
\vspace{2mm}

(c) All  BGN extensions   on $C$ of $F$  of type $(r,d,k)$ are classified  by linearly independent  $k$-tuples of elements of $\Ext^1(F,\OO_C)$ modulo to the action of $\GL(k)$.
Hence, they are parametrized by the variety $\Gr(k,\Ext^1(F,\OO_C))$. By Serre duality we have
$\Ext^1(F,\cO_C) \simeq \Ext^1(F\otimes \omega_C,\omega_C) \simeq H^0(F\otimes \omega_C)^* \simeq H^1(F^*)$.
\end{proof}

\begin{proposition}
\label{PROP:restriction}
Let $(C, \w)$ be a polarized nodal curve.  Let $F$ be a locally free sheaf on $C$ of rank $r-k$ and degree $d$.  Let $F_i$ be the restriction to the component $C_i$ and $d_i = \deg(F_i)$. 
If $F_i$ is semistable,  $d_i>0$ and $0<k  \leq d_i + (r-k)(g_i -1)$, there exists a non empty open subset $U \subset \Gr(k,H^1(F^*))$ such that any $u \in U$ defines a BGN extension  of $F$ whose restriction to the curve $C_i$ is  a BGN extension on $C_i$  of $F_i$ of type $(r,d_i,k)$. 
\end{proposition}

\begin{proof}
Consider a BGN extension    corresponding to ${\underline e} \in \Ext^1(F,V \otimes \OO_C) = V \otimes H^1(F^*)$:
$$\underline{e}:\qquad 0 \to  V \otimes \OO_C \to E \to F \to 0.$$
Since  $F$ is locally free, $\Tor^1_{\OO_C}(\OO_{C_i},F)= 0$, so  by tensoring with $\OO_{C_i}$ we get the exact sequence on $C_i$:
$$0 \to V \otimes \OO_{C_i} \to E_i \to F_i \to 0.$$
Its  corresponding   extension  class is an element ${\underline e}^i \in \Ext^1(F_i,V \otimes \OO_{C_i}) = V \otimes H^1(F_i^*)$, where  $F_i^* = \cHom_{\OO_{C_i}}(F_i,\OO_{C_i})$. 
This gives us a natural map
$\alpha_i:  V \otimes H^1(F^*) \to  V \otimes H^1(F_i^*)$, sending ${\underline e} \to {\underline e}^i$.
As $F$ is locally free, we have $F_i^* \simeq F^* \otimes \OO_{C_i}$. Hence we have a restriction map
$\rho_i \colon F^* \to F_i^*$ on the component $C_i$. 
We claim  that 
 $\alpha_i= \id_V \otimes  H^1(\rho_i)$, where $H^1(\rho_i) \colon H^1(F^*) \to H^1(F_i^*)$  is the map induced by  $\rho_i$. 
In fact,   consider  the following  commutative  diagram:
$$
\xymatrix{
 \Hom(F,V \otimes O_C)   \ar[d] \ar@{^{(}->}[r]  &  \Hom(F,E)  \ar[d] \ar[r] &   \Hom(F,F)  \ar[d] \ar[r]^-{\delta}   & V \otimes H^1(F^*)  \ar[d] \\
 \Hom(F_i,V\otimes \OO_{C_i}) \ar@{^{(}->}[r] & \Hom(F_i,E_i)  \ar[r] &  \Hom(F_i,F_i)  \ar[r]^-{\delta_i}  & V \otimes H^1(F_i^*) 
 }
$$
where the vertical arrows are induced by the restriction to $C_i$.   As the extension  class 
$\underline e$ is the image by $\delta$ of $\id_F$ and ${\underline e}^i$ is the image by $\delta_i$ of $\id_{F_i}$, the claim follows. 
\hfill\par
Since  $h^0(F^*_i) = 0$, 
 to prove that  ${\underline e}^i$ defines a BGN extension on $C_i$ it is enough to verify that  $\{ {\underline e}^i_1, \dots {\underline e}^i_k \}$ are linearly independent vectors in $H^1(F_i^*)$.  $\Gr(k,H^1(F_i^*))\not= \emptyset$ since we assume $k \leq d_i + (r-k)(g_i -1)$.
Note that the restriction map $H^1(\rho_i) \colon H^1(F^*) \to H^1(F_i^*)$ is  a linear surjective map. The image of a $k$-dimensional subspace $W$ of $H^1(F^*)$ has dimension $k$ if and only if it has trivial intersection with the kernel of $H^1(\rho_i)$. This happens for $W$ in an open subset $U$ of $\Gr(k,H^1(F^*))$ since we are assuming $k\leq h^1(F_i^*)$. Hence $H^1(\rho_i)$ induces a rational surjective map between Grassmannian varieties
$$
\xymatrix{
\Gr(k,H^1(F^*)) \ar@{-->}[r]^-{A_i} & \Gr(k,H^1(F_i^*))
}$$
which is defined on $U$. 
Each extension class in $U$ gives, by restriction to the curve $C_i$, a BGN extension class on $C_i$ of $F_i$ of type $(r,d_i,k)$. 
\end{proof}


\section{Coherent systems and BGN extensions}
\label{SEC:3}

Let $(C, \w)$ be a  polarized nodal  curve with $\gamma$ irreducible smooth components of genus $g_i \geq 2$ and $\delta$ nodes. In this section we   will study  moduli spaces $\Gwa(r,d,k)$ with $k<r$  by using   BGN extensions. 
Recall that  $\alpha >0$ is a {\it critical value}  for coherent systems of type $(r,d,k)$  if  it is numerically possible to have a  proper coherent subsystem  $(F,W)$ of $(E,V)$ with 
$$\mu_{\wa}(F,W)= \mu_{\wa}(E,V), \quad \mu_{\w}(F) \not= \mu_{\w}(E).$$ 
In our first result we prove that,  as in the smooth case,  on a polarized curve $(C,\w)$ there are only finitely many  spaces $\Gwa(r,d,k)$ for fixed $(r,d,k)$.
 
\begin{lemma}
\label{LEM:pcritici}
Let $(C,\w)$ be a polarized nodal curve. Fix $M>0$,  there are up to finitely many critical value in $(0,M)$ for coherent systems of type $(r,d,k)$ with   $k < r$: 
$$ 0 < \alpha_1 < \dots < \alpha_{i_M}<M.$$
Moreover, within the intervals $(0, \alpha_1)$, $(\alpha_i, \alpha_{i+1})$ and $\left(\alpha_{i_M},M\right)$ the property of $(\w,\alpha)$-stability is independent of $\alpha$.
\end{lemma}
\begin{proof}
Denote by $w_m=\min\{w_i \,|\, i=1,\dots,\gamma\}$. The critical value of $\alpha$ for coherent systems of type $(r,d,k)$ can be written as 
$$\alpha = \frac{rd'-r'd}{r'k-rk'} \quad \mbox{ with }\quad  rk' \not= r'k$$
where
\begin{enumerate}
\item{} $k'\in \mathbb N$,  $0 \leq k' \leq k$;
\item{} $r' = \sum_{i=1}^{\gamma}w_ir_i'$, with 
$r_i' \in \mathbb N$, $0 \leq  r_i' \leq r_i\leq r/w_{m}$;
\item{} $d' + r'\chi(\OO_C) \in \mathbb Z$.
\end{enumerate}
Note that, by (1) and (2),   there are only finite possibilities for $k'$ and $r'$. Moreover,  by (3), once we fix $r'$, $d'$   varies in a discrete set.  Since $\alpha \in (0, M)$, this gives finitely values for $d'$. 
\hfill\par
In order to  prove the last assertion,  let $\alpha, \beta \in (0,M)$ with $\alpha,\beta$ not critical. Assume that there exists $(E,V)$  which is $(\w,\alpha)$-stable but it is not $(\w,\beta)$-stable. It is enough to prove that between $\alpha$ and $\beta$ there is a critical value. Let  $(F,W)$ be a coherent subsystem of $(E,V)$ such that $\mu_{(\w,\beta)}(F,W) \geq  \mu_{(\w,\beta)}(E,V)$.
If $\mu_{(\w,\beta)}(F,W) = \mu_{(\w,\beta)}(E,V)$, as $\beta$ is not a critical value, we have $\mu_{\w}(F)= \mu_{\w}(E)$ and so $\frac{\dim W}{\wrank(F)} = \frac{\dim V}{\wrank(E)}$. This implies that $\mu_{(\w,\alpha)}(F,W)= \mu_{(\w,\alpha)}(E,V)$, which contradicts the $(\w,\alpha)$-stability of $(E,V)$. So we can assume
$$\mu_{(\w,\beta)}(F,W) - \mu_{(\w,\beta)}(E,V) > 0 \quad \quad \mu_{(\w,\alpha)}(F,W) - \mu_{(\w,\alpha)}(E,V) < 0.$$
We claim that there exists $t \in (0,1)$ such that $t\alpha + (1-t) \beta$ is a critical value. 
Let $t \in (0,1)$, then we have:  
$$\mu_{(\w,t\alpha +(1-t)\beta)}(F,W) - \mu_{(\w,t\alpha+ (1-t)\beta)}(E,V)  =$$
$$t [\mu_{(\w,\alpha)}(F,W)- \mu_{(\w,\alpha)}(E,V)] + (1-t)[\mu_{(\w,\beta)}(F,W)- \mu_{(\w,\beta)}(E,V)].$$
So there exists $t_0 \in (0,1)$ such that $$\mu_{(\w,t_0\alpha +(1-t_0)\beta)}(F,W) - \mu_{(\w,t_0\alpha+ (1-t_0)\beta)}(E,V)= 0.$$
If $\mu_{\w}(F)=\mu_{\w}(E)$ we get a contradiction as before, whereas if $\mu_{\w}(F) \not=  \mu_{\w}(E)$, then $t_0 \alpha + (1-t_0) \beta  \in (0,M)$ is a critical value. 
\end{proof}

\begin{proposition}
\label{PROP:alphaL}
Let $(C,\w)$ be a polarized nodal curve. 
Let $r >0$ and  $d$ be rational numbers and  $k$ be an integer such that $0 < k < r$.
Then the moduli space $\Gwat(r,d,k)$ is  empty for any  $\alpha > \frac{d+r\lambda_{\w}}{r-k}$ and 
$\Gwa(r,d,k)$ is empty for  any $\alpha  \geq  \frac{d+r\lambda_{\w}}{r-k}.$
Moreover, if $\Gwa(r,d,k) \not= \emptyset$, then $d> -r\lambda_{\w}$.
\end{proposition}

\begin{proof} 
Let $(E,V)$ be a coherent system of type $(r,d,k)$ which is $\wa$-semistable.  We have
$\wrank(E) =  r$ and 
$\wdeg(E) =  \chi(E) -\wrank(E)\chi(\OO_C) = d$. 
Since $ k< r$ the evaluation map $\ev_V \colon V \otimes \OO_C \to E$ is not surjective, let $F$ be its image. It is a non zero sheaf and it is of depth one so $(F,V)$ is a generated coherent system, which is a proper coherent subsystem  of $(E,V)$. 
Let $s=\wrank(F)$, by Lemma \ref{LEM:generatedsystem}, we have $0<s\leq k$.   By $\wa$-semistability  of $(E,V)$ we have $\mu_{\wa}(F,V) \leq \mu_{\wa}(E,V)$, i.e.
\begin{equation}
\label{dis1}
\frac{\wdeg(F)}{s} + \alpha \frac{k}{s} \leq \frac{d}{r} + \alpha \frac{k}{r}.
\end{equation}
This is equivalent to the following inequality:
$$k\alpha\left(\frac{1}{s}-\frac{1}{r}\right) \leq \frac{d}{r} - \frac{\wdeg(F)}{s}.$$
As $0 < s \leq k < r$, by Lemma \ref{LEM:wdegree} and Definition \ref{DEF:lambdaw} we can write
$$ \alpha \leq \frac{sd}{k(r-s)}- \frac{r\wdeg(F)}{k(r-s)}\leq\frac{sd}{k(r-s)}- \frac{r\sum_{i=1}^{\gamma}\deg(F_i)}{k(r-s)}+ \frac{rs\lambda_{\w}}{k(r-s)}.$$
Since $(F,V)$ is generated, by Lemma \ref{LEM:generatedsystem}, we have $\deg(F_i) \geq 0$ so we obtain
\begin{equation}
0 \leq  \alpha \leq \frac{s}{k(r-s)}(d+r\lambda_{\w})
\end{equation}
This implies that $d+r\lambda_{\w}\geq 0$. Hence, since $r-s \geq r-k$ we obtain
\begin{equation}
\alpha \leq \frac{d+r\lambda_{\w}}{r-k}.
\end{equation}
If we assume the existence of $(E,V)$  which is $\wa$-stable then we would get 
\begin{equation}
\alpha < \frac{d+r\lambda_{\w}}{r-k},
\end{equation}
and $d > -r\lambda_{\w}$.
\end{proof}

\begin{corollary}
\label{COR:goodcase}
Let $(C,\w)$ be  a polarized nodal curve with $\w$ good.  If  $\Gwa(r,d,k) \neq \emptyset$, then $d>0$ and $\alpha\in (0,d/(r-k))$.
\end{corollary}

We point out  that these bounds are exactly the same  which hold for a smooth curve (see \cite{BDGW}) and an irreducible nodal curve (see \cite{Bho}). 
\vspace{2mm}

When $C$ is  a smooth or a nodal irreducible curve, moduli spaces of $\alpha$-stable coherent systems with $k<r$ are closely related  to  BGN extensions. We would like to show that this connection holds also on nodal reducible curve.  Note that, given a BGN extension
$$\underline{e}:\quad  0 \to V \otimes \OO_C \to E \to F \to 0$$
of type $(r,d,k)$ we have a coherent system $(E,V)$ of type $(r,d,k)$, such that  $\ev_V$ is injective. We will call $(E,V)$ the coherent system defined by $\underline e$.

\begin{lemma}
\label{LEM:posdegree}
Let $(C,\w)$ be  a polarized nodal curve.
Let $r$ and  $d$ be rational numbers and  $k$ be an integer such that $0 < k < r$. Let $(E,V)$ be a $\wa$-semistable  coherent system of type $(r,d,k)$. If the evaluation map $ev_V \colon V \otimes \OO_C \to E$ is injective, then either $d > 0$ and $\alpha \leq d/(r-k)$ or $\alpha=d=0$.
\end{lemma}

\begin{proof}
 Since $ev_V$ is injective and $k < r$, $V \otimes \OO_C$  is a  proper subsheaf of $E$. 
The pair $(V \otimes \OO_C, V)$ is a proper coherent subsystem of $(E,V)$. Since it is   $\wa$-semistable  we have:
$$\mu_{\wa}(V \otimes \OO_C,V) \leq \mu_{\wa}(E,V),$$
that is
$$ \alpha \leq \frac{d}{r} + \alpha\frac{k}{r},$$
 since $0 < k < r$ this implies  $\alpha \leq d/(r-k)$ and  hence $d >0$ unless $\alpha=0$ (which implies $d=0$).
\end{proof}

In light of these facts and the discussion made above, it is natural to fix a good polarization $\w$ on $C$. Hence, from now on,  $\w$ will be a good polarization. Then, by Proposition \ref{PROP:alphaL}, all critical values for coherent systems of type $(r,d,k)$ are in the interval $\left(0,\frac{d}{r-k}\right)$. We will denote by $\alpha_L$ the biggest among the critical values. We will denote by $\GwL
(r,d,k)$ the {\it "limit"} moduli space of coherent systems of type $(r,d,k)$ which are $\wa$-stable for $\alpha \in (\alpha_L,\frac{d}{r-k})$. 
 
\begin{lemma}
\label{LEM:alphaI}
Let $C$  be a  nodal curve and $\w$ be a good polarization. 
Let $r >0$ and  $d > 0 $ be rational numbers and  let $k$ be an integer with $0 < k < r$. There exists $ 0 < \alpha_I < \frac{d}{r-k}$ such that  if  $\alpha  \in (\alpha_I, \frac{d}{r-k})$ 
 for any $\wa$-semistable coherent system $(E,V)$ the 
 evaluation map $ev_V \colon  V \otimes \OO_C \to E$ is injective. 
 \end{lemma}
 \begin{proof}
Let  $0 < \alpha < \frac{d}{r-k}$ and let  $(E,V)$ be a coherent system  which is $\wa$-semistable.  Assume that 
the evaluation map $ev_V$ is not injective, so we have an exact sequence of sheaves:
$$ 0  \to N \to V \otimes \OO_C \to F \to 0,$$
where $N$ and $F$ are non zero sheaves satisfying the following  properties: 
\hfill\par
- $N$   is a  proper subsheaf  of $V \otimes \OO_C$, so it is  a depth one sheaf too;   let $\wrank(N) = \eta$,  then  we have $\eta \geq w_m$, where $w_m = \min(w_1,\dots ,w_{\gamma})$.
 \hfill\par
- $F$ is   proper subsheaf of $E$, so $F$ is a depth one sheaf too; let   $\wrank(F) = s$,  then we have  $$s= k - \eta, \quad  w_m \leq s \leq k - w_m.$$    
The pair $(F,V)$ is  a generated  coherent system,  which is a  proper subsystem of $(E,V)$.  As this is $\wa$-semistable, we have $\mu_{\wa}(F,V) \leq \mu_{\wa}(E,V),$ which implies
\begin{equation}
\label{dis3}
\frac{\wdeg(F)}{s} + \alpha \frac{k}{s} \leq \frac{d}{r} + \alpha \frac{k}{r}.
\end{equation}
As $s <k <r$, we can proceed as  in the proof of Proposition  \ref{PROP:alphaL} and we obtain:
$$\alpha \leq \frac{sd}{k(r-s)}- \frac{r\wdeg(F)}{k(r-s)}\leq\frac{sd}{k(r-s)}- \frac{r\sum_{i=1}^{\gamma}\deg(F_i)}{k(r-s)}+ \frac{rs\lambda_{\w}}{k(r-s)}$$
Since $(F,V)$ is generated, then $deg(F_i) \geq 0$. Then, as $\w$ is good, we have $\lambda_{\w}=0$. Hence we obtain
$$\alpha \leq \frac{sd}{k(r-s)}.$$
Finally as $s \leq k-w_m$ we get:
\begin{equation}
    \label{EQ:alpha1}
\alpha \leq \frac{(k-w_m)d}{k(r-k+w_m)}=\alpha_I.   
\end{equation}
Note that $ 0 < \alpha_I < \frac{d}{r-k}$. 
We can conclude that  if $\alpha \in (\alpha_I, \frac{d}{r-k})$  the evaluation map is injective for any $\wa$-semistable $(E,V)$.  
\end{proof}

\begin{remark}
Notice that,  without assuming $\w$ good, the proof of Lemma \ref{LEM:alphaI} still works and yields the bound
$$\alpha \leq (d+r\lambda_{\w})\frac{(k-w_m)}{k(r-k+w_m)}=\alpha_I'.$$
Hence, if $\alpha>\alpha_I'$, every $\wa$-semistable coherent system has $\ev_V$ injective. Nevertheless, by Lemma \ref{LEM:posdegree}, for such coherent systems one has necessarily $\alpha<d/(r-k)$. Unfortunately, for a general polarization, it can happens that $\alpha_I'>d/(r-k)$.
\end{remark}

\begin{lemma}
\label{LEM:alphaT}
Let $C$ be  a nodal curve and $\w$ be a good polarization on it. 
Let $r >0$ and  $d > 0 $ be rational numbers and   let $k$ be an integer with  $0 < k < r$.
There exists $ 0 < \alpha_T < \frac{d}{r-k}$ such that  if  $\alpha  \in (\alpha_T, \frac{d}{r-k})$ 
 for any $\wa$-semistable coherent system $(E,V)$ we have an exact sequence: 
\begin{equation}
 0 \to V \otimes \OO_C \to E \to F \to 0,
\end{equation}
where  $F$ is a sheaf of depth one. 
\end{lemma}

\begin{proof}
Let $\alpha \in (\alpha_I, \frac{d}{r-k})$. Let $(E,V)$ be a coherent system which is $\wa$-semistable. By Lemma \ref{LEM:alphaI} the evaluation map $V \otimes \OO_C \to E$ is injective. 
So we have an exact sequence of sheaves:
$$ 0 \to V \otimes \OO_C \to E \to F \to 0.$$
Assume that $F$ is not of depth one. Let $T$ be the maximal torsion subsheaf of $F$. Then $\dim  T = 0$, $\wrank(T) = 0$ and $\wdeg(T) = \chi(T) = \tau \geq 1$. 
 The quotient $F_0 = F/T$ is a depth one sheaf on $C$. We have the following commutative diagram:

\begin{equation}
\xymatrix{
V\otimes \OO_C \ar@{^{(}->}[r] \ar@{=}[d] &
    E_0 \ar@{->>}[r]\ar@{^{(}->}[d] &
    T \ar@{^{(}->}[d] \\
V\otimes \OO_C \ar@{^{(}->}[r] &
    E \ar@{->>}[r]\ar@{->>}[d] &
    F \ar@{->>}[d] \\
 &
    F_0\ar@{=}[r] & 
    F_0
}
\end{equation}

 The pair $(E_0,V)$ is a proper  coherent subsystem of $(E,V)$ of type $(k,\tau,k).$
 By $\wa$-semistability of $(E,V)$  we have:
 $$\mu_{\wa}(E_0,V) \leq \mu_{\wa}(E,V),$$
 that is:
 $$\frac{\tau}{k} + \alpha  \leq \frac{d}{r} + \alpha \frac{k}{r}.$$
 This implies:
 $$ \alpha\left(1 - \frac{k}{r}\right) \leq \frac{d}{r}- \frac{\tau}{k}.$$
Since $k < r$ and $\tau \geq 1$,  we have:
$$ \alpha \leq \frac{d}{r-k}-\frac{r\tau }{k(r-k)} \leq \frac{d}{r-k} - \frac{r}{k(r-k)}=\tilde{\alpha}_T.$$
Note that $\tilde{\alpha}_T < \frac{d}{r-k} $. 
Hence we can conclude by defining $\alpha_T$ to be the maximum between $\alpha_I$ and $\tilde{\alpha}_T$.
\end{proof}

Let  $(E,V)$ be  a coherent system   of type $(r,d,k)$ with $d>0$ and $0 < k < r$ defining an exact sequence as follows 
 $$ 0 \to V \otimes \OO_C \to E \to F \to 0,$$
 where  $F$ is a sheaf of depth one. 
 Let $\underline{e} \in \Ext^1(F, V \otimes \OO_C)$ be the  corresponding extension class.
 We have $\Ext^1(F,V \otimes \OO_C) \simeq V \otimes \Ext^1(F,\OO_C)$, so $\underline{e}= (\underline{e}_1, \dots, \underline{e}_k)$, 
 $\underline{e}_i \in \Ext^1(F,\OO_C)$. 
 
\begin{lemma}
\label{LEM:extension}
With the above notations, let $(E,V)$ be $\wa$-semistable, with $\alpha \leq \frac{d}{r-k}$,  then we have:
\begin{enumerate}
    \item{}  if  $\underline{e} = \underline{0}$ then $\alpha = \frac{d}{r-k}$; 
    \item{}   if $\alpha < \frac{d}{r-k}$, then $\underline{e}_1, \dots, \underline{e}_k$ are linearly independent in $\Ext^1(F, \OO_C)$. 
    In particular, $k \leq h^1(F^*).$
\end{enumerate}
\end{lemma}
\begin{proof}
(1) Assume that  $\underline{e} = \underline{0}$, then  we have  
 $(E,V) \simeq (V \otimes \OO_C, V) \oplus (F,0)$.
 In particular, 
 as  $(E,V)$ is $\wa$-semistable, we must  have:
$$\frac{\wdeg(F)}{\wrank(F)} \leq \frac{d}{r} + \alpha \frac{k}{r},$$
that is 
$$\frac{d}{r-k} \leq \frac{d}{r} + \alpha \frac{k}{r},$$
which implies:
$\alpha \geq \frac{d}{r-k}$. Since we assumed $\alpha \leq \frac{d}{r-k}$, we get the equality. 
\hfill\par
(2) Assume that 
$\underline{e}_1, \dots, \underline{e}_k$ are linearly dependent. After a base change, we can assume that $\underline{e} = (\underline{e}_1,\dots, \underline{e}_l,\underline{0}, \dots\underline{0})$, with $1 \le l < k$ and $\underline{e}_1,\dots, \underline{e}_l$ linearly independent.
This implies that 
$$(E,V)  \simeq  (E_1,V_1) \oplus (V_2\otimes \OO_C,V_2),$$
with $\dim V_1= l$ and  $\dim V_2 = k-l$.
We have the following  relation between $\wa$-slopes: 
$$\mu_{\wa}(E,V) = \frac{r-(k-l)}{r} \mu_{\wa}(E_1,V_1) + \frac{k-l}{r}\mu_{\wa}(V_2\otimes \OO_C,V_2).$$ 
So we have the following cases:
\begin{enumerate}
    \item{} $\mu_{\wa}(E,V) =  \mu_{\wa}(E_1,V_1)=  \mu_{\wa}(V_2\otimes \OO_C,V_2),$
    \item{} either $(E_1,V_1)$ or $(V_2\otimes \OO_C,V_2)$ destabilizes $(E,V)$.
\end{enumerate}
     
In the first case we obtain $\alpha = \frac{d}{r-k}$, the second one cannot occur since $(E,V)$ is $\wa$-semistable. 
\end{proof}

\begin{lemma}
\label{LEM:alphaS}
Let $C$ be a nodal curve and let $\w$ be a good polarization. Let $r,d >0$ rational and $k$ an integer with $0 < k < r$. Then, there exists $ \alpha_S > \alpha_T$ and $\alpha_S < \frac{d}{r-k}$ such that if  $\alpha  \in (\alpha_S, \frac{d}{r-k})$, any $\wa$-semistable coherent system $(E,V)$ defines a BGN extension
\begin{equation}
     \label{BGN:EV}
0 \to V \otimes \OO_C \to E \to F \to 0,
\end{equation}
of type $(r,d,k)$ with $F$ ${\underline w}$-semistable and $h^0(F^*)=0$.
\end{lemma}
\begin{proof}
Let $(E,V)$ be $\wa$-semistable with $\alpha > \alpha_T$. Then, 
by Lemma \ref{LEM:alphaT} and Lemma \ref{LEM:extension}, it defines a BGN extension with $F$ of depth one.  
We have to prove that $F$ is $\w$-semistable for $\alpha$ big enough. 
Let $F' \subset F$ be a  proper subsheaf of $F$ such that the quotient $Q= F /F'$ is a sheaf of depth one. 
We denote by $s' = \wrank(F')$ and $d' = \wdeg(F')$.
Then we have $w_m \leq s' < r-k$, where $w_m = \min\{w_1,\dots,w_{\gamma}\}$.
From the commutative diagram:
$$
\xymatrix{
V\otimes \OO_C \ar@{^{(}->}[r] \ar@{=}[d] &
    E' \ar@{->>}[r]\ar@{^{(}->}[d] &
    F' \ar@{^{(}->}[d] \\
V\otimes \OO_C \ar@{^{(}->}[r] &
    E \ar@{->>}[r]\ar@{->>}[d] &
    F \ar@{->>}[d] \\
 &
    Q\ar@{=}[r] & 
    Q
}
$$
we obtain a coherent system $(E',V)$ of type
$(s'+k,d',k)$ which is a proper subsystem of $(E,V)$.
Since this  is $\wa$-semistable we have:
$$\mu_{\wa}(E',V) \leq \mu_{\wa}(E,V),$$
that is
$$ \frac{d'}{k+s'} + \alpha\frac{k}{k+s'} \leq \frac{d}{r}+ \alpha\frac{k}{r}.$$
We can write the above inequality as  follows
$$ \frac{s'}{k+s'} \mu_{\underline w}(F') + \alpha\frac{k}{k+s'} \leq  \frac{s'}{k+s'}\mu_{\underline w}(F)  +  \left(\frac{r-k}{r} - \frac{s'}{k+s'}\right) \mu_{\underline w}(F) + \alpha\frac{k}{r}, $$
equivalently
$$ \frac{s'}{k+s'}(\mu_{\underline w}(F') - \mu_{\underline w}(F)) \leq \alpha \left(\frac{k}{r}-\frac{k}{k+s'}\right) + \frac{k(r-k-s')}{r(k+s')}\mu_{\underline w}(F),$$
as $s' > 0$ and  $k + s' > 0$,  we get
$$ (\mu_{\underline w}(F') - \mu_{\underline w}(F)) \leq \frac{k(r-k-s')}{rs'}(\mu_{\underline w}(F) - \alpha).$$
Note that $\mu_{\underline w}(F)= \frac{d}{r-k}$  and $s' \geq w_m > 0$, so we obtain:
$$ \mu_{\underline w}(F') - \mu_{\underline w}(F) \leq \frac{k(r-k-w_m)}{rw_m}\left(\frac{d}{r-k} - \alpha\right).$$
Let  $s' = \sum_{i=1}^{\gamma}w_i s_i'$,  as $s'< r-k$, then    we  have $0 \leq s_i'\leq (r-k)/w_m$.  In particular, 
for a fixed polarization $\underline w$, there are finitely many values  for $s'$. 
As $d'= \chi(F') - s' \chi(\OO_C)$,  the possible value of $d'$ lies in a discrete set whose intersection with any bounded subset is finite.  In particular, the possible values which the difference $\mu_{\underline w}(F') - \mu_{\underline w}(F)$ can assume in $[0,1]$ are finite.  
So there exists $q > 0$ such that 
$\mu_{\underline w}(F') - \mu_{\underline w}(F) \leq q $ implies $\mu_{\underline w}(F') - \mu_{\underline w}(F) \leq 0$. 
If we choose 
$$ \alpha > \frac{d}{r-k} - q \frac{rw_m}{k(r-k-w_m)},$$
then we have $$\mu_{\underline w}(F') - \mu_{\underline w}(F) \leq q 
$$
and so $\mu_{\w}(F') \leq \mu_{\w}(F).$
We set $\alpha_S = \frac{d}{r-k} - q \frac{rw_m}{k(r-k-w_m)}$. It is easy to see that $\alpha_S  > \alpha_T$. 
\vspace{2mm}

Finally, as $F$ is $\w$-semistable with $\wdeg(F) = d>0$ and $\w$ is good, by Lemma \ref{LEM:h0dual} we have that $h^0(F^*)=0$.
\end{proof}


The consequence of all previous technical result is the following Theorem. 

\begin{theorem}
\label{THM:BGNext}
Let $C$ be a nodal curve and $\w$ a good polarization on it.  
Let $r >0$ and  $d > 0 $ be rational numbers and   let $k$ be an integer with  $0 < k < r$. If $\GwL(r,d,k)$ is not empty, we have a morphism
$$\eta:\GwL(r,d,k) \to {\mathcal U}_{(C,\w)}(r-k,d)$$
sending a coherent system $(E,V)$ to $[\coker(\ev_V)]$.
\end{theorem}

\begin{proof}

First of all we point out that $\coker(\ev_V)$ does not depend on the isomorphism class of $(E,V)$. By Lemma \ref{LEM:extension} and Lemma \ref{LEM:alphaS} if $\alpha \in \left(\alpha_S, \frac{d}{r-k}\right)$ any   $\wa$-stable coherent system $(E,V)$ defines a BGN extension of type $(r,d,k)$, where  $F = \coker ev_V$  is a depth one sheaf  which is $\w$-semistable, with $\wrank(F) = r-k$ and $\wdeg(F) = d$. 
If $\alpha_S < \alpha_L$, then this holds for any 
    $(E,V) \in \GwL(r,d,k)$.
Let  $\alpha_S > \alpha_L$, then   for  any $(E,V) \in \GwL(r,d,k)$,   $\wa$-stability does not change if $\alpha $ varies in $(\alpha_L, \frac{d}{r-k})$ by Lemma \ref{LEM:pcritici}. As
Hence, we obtain that $\eta$  is well defined. 
\hfill\par
To prove that $\eta$ is  a morphism  we consider a family of $\wa$-stable coherent systems of type $(r,d,k)$ parametrized by a variety  $S$. 
In particular,  in our hypothesis, we have a coherent sheaf ${\mathcal E}$ on $S \times C$, which is  flat on $S$ such that $E_s = \mathcal E \otimes \cO_{s \times C}$  is a depth one sheaf, a vector bundle ${\mathcal V}$ on $S$ of rank $k$ and a  map of sheaves
$\xi \colon {\pi_S}^*\cV \to \mathcal E$, where $\pi_S \colon S \times C \to S$ is the projection.  Moreover, for any $s$
the map $\xi_{\vert s \times C} \colon V_s \otimes \cO_{s \times C} \to E_s$  is injective and its  cokernel $F_s$ is a $\w$-semistable depth one sheaf with $\wrank(F_s)=r-k$  and $\wdeg(F_s)=d$.  
If we assume that $S$ is irreducible and reduced, we  can prove that $\xi$ is an injective map of sheaves on $S \times C$.

This implies that  its cokernel ${\mathcal F}$ is a coherent sheaf on $S \times C$, which is flat on $S$ too and ${\mathcal F}_s = F_s$ defined as above.
So ${\mathcal F}$ defines a  family of 
$\w$-semistable sheaves of depth one of $\w$-rank $(r-k) $  and $\w$-degree $d$, parametrized by $S$. Hence we have a natural morphism $S \to {\mathcal U_{(C,\w)}}(r-k,d)$, sending $s \mapsto [F_s]$. This proves that $\eta$ is a morphism. 
\end{proof}

\begin{theorem}
\label{THM:cs}
Assume that we are in the same hypothesis of Theorem \ref{THM:BGNext}. Then the image of $\eta$ contains the subscheme $\mathcal{U}_{(C,\w)}^s(r-k,d)$ parametrizing $\w$-stable sheaves and the fiber of $\eta$ over $F\in\mathcal{U}_{(C,\w)}^s(r-k,d)$ is $\Gr(k,H^1(F^*))$.
\end{theorem}
\begin{proof}
Let $F$ be a $\w$-stable sheaf of depth one on $C$  with  $\wrank(F)=r-k$ and $\wdeg(F)=d$. Let    ${\underline e} \in \Ext^1(F,V \otimes \OO_C)$ be the class of a BGN  extension of type $(r,d,k)$: 
\begin{equation}
\label{ext1}
0 \to V \otimes \OO_C \to E \xrightarrow{\beta} F \to 0.
\end{equation}
Let  $(E,V)$ be  the  coherent system on $C$ defined by ${\underline e}$.
We will prove that $(E,V)$ is $\wa$-stable for $\alpha\in (\alpha_L,d/(r-k))$.
At this end,  let $(E',V')$  be a proper coherent subsystem  of $(E,V)$ with 
$V' = V \cap H^0(E')$, it is enough to see that 
$$\mu_{\wa}(E',V') < \mu_{\wa}(E,V),$$
for $\alpha$  sufficiently  close to $\frac{d}{r-k}$. 
For this reason, we set $\alpha = \frac{d -\epsilon}{r-k}$, with $\epsilon > 0$. \hfill\par

We consider the restriction of $\beta$ of \eqref{ext1} to $E'$: $\beta|_{E'} \colon E' \to F$.  
We  will distinguish  two cases  depending on  the fact the $\beta|_{E'}$ is the zero map.
\hfill\par
{\bf Case (a)}: $\beta|_{E'} = 0$.  Then $E'$  is a  nonzero subsheaf of $V \otimes \OO_C$ with  $\wrank(E')= r' \leq k < r$. As $\w$ is good we have that
$V \otimes \OO_C$ is ${\underline w}$-semistable (see \cite{BFPol}),  then we have:
$$\mu_{\underline w}(E') \leq \mu_{\underline w}(V \otimes \OO_C) = 0,$$
which implies $ \wdeg(E') = d'  \leq 0$.
Let  $k' = \dim V'$, then $0 \leq k' \leq k$.  If $k' > 0$, since $V' \subseteq V$, then we have an injective map $V' \otimes \OO_C \to E'$ which implies $k' \leq r' \leq k < r$. 
So we have:
$$ \mu_{\wa}(E',V')= \frac{d'}{r'} + \alpha\frac{k'}{r'} \leq \alpha.$$
As  we have chosen $\alpha= \frac{d-\epsilon}{r-k}$, it follows that
$$\mu_{\wa}(E,V) = \frac{d}{r} + \alpha \frac{k}{r}= \alpha + \frac{\epsilon}{r} > \alpha.$$
These two inequalities  allow us to   conclude: 
$$\mu_{\wa}(E',V') \leq \alpha < \mu_{\wa}(E,V).$$

{\bf Case (b)}: $\beta_{\vert E'} \neq 0$.
Then we set  $G' = \Ker (\beta_{\vert E'})$  and 
$F'= \Ima (\beta_{\vert E'})$.  They  are  depth one sheaves on $C$, $G'$ is a subsheaf of $V \otimes \OO_C$ and 
$F'$ is a non zero  subsheaf of $F$,   
they  fit into the following exact sequence:
$$0 \to G' \to E' \xrightarrow{\beta{|_E'}}
F' \to 0.$$
We distinguish two cases depending on the fact that $F'$   coincide with  $F$ or not. 
\hfill\par
{\bf (b1)}: $F'$ is a proper subsheaf  of $F$. 
Let $S_F= \{ \mu_{\w}(F'')\,|\, F'' \mbox{ proper subsheaf of } F \}$. Since $F$ is $\w$-stable,  the intersection $S_F \cap (0, \mu_{\w}(F))$ is finite.   
 So    we  can choose $\delta >0$  such that for any proper 
subsheaf $F'$ of $F$ we have:
$$\mu_{\underline w}(F') \leq \mu_{\underline w}(F) -\frac{\delta}{r-k}.$$ 
We set
$$\wrank(G') = l' \quad   \wrank(E') = r' \quad     \wrank(F') =  r'-l',$$
with $ 0 \leq l' < r'$,  and 
$$\wdeg(E')=  d' = \wdeg(G') + \wdeg(F').$$
As $G' \subseteq V \otimes \OO_C$ and $\w$ is good, by $\w$-semistability of $V \otimes \OO_C$  we have
$\wdeg(G') \leq 0$, hence $\wdeg(E') \leq \wdeg(F')$.
So we have:
\begin{equation}
\label{eq:b1}
\mu_{\underline w}(E') = \frac{\wdeg(E')}{r'} \leq 
\frac{\wdeg(F')}{r'} \leq \frac{r'-l'}{r'} \mu_{\underline w}(F') \leq \frac{r'-l'}{r'}\left[\mu_{\underline w}(F) - \frac{\delta}{r-k}\right].
\end{equation}
We set  $k' = \dim V'$,  as $V' \subseteq V$, then  we have an injective map  $V' \otimes \OO_C  \to G'$, so $k' \leq l'< r'$.
Finally, $(E',V')$  is a  proper coherent subystem of $(E,V)$ of type $(r',d',k')$. 
By using Equation \eqref{eq:b1} we have:
$$\mu_{\wa}(E',V') - \mu_{\wa}(E,V)= \mu_{\underline w}(E')  - \mu_{\underline w}(E) + \alpha\left(\frac{k'}{r'} - \frac{k}{r}\right) \leq $$
$$\leq \frac{r'-l'}{r'}\left[\mu_{\underline w}(F) - \frac{\delta}{r-k}\right] - \mu_{\underline w}(F)\frac{r-k}{r} + \alpha\left(\frac{k'}{r'}-\frac{k}{r}\right) \leq$$
$$ \leq \mu_{\underline w}(F)\frac{k'-l'}{r'}  - \frac{\delta}{r-k}\left(1 - \frac{l'}{r'}\right)+  \frac{\epsilon}{r-k}\left(\frac{k}{r}-\frac{k'}{r'}\right).$$
Since  $k' \leq l'$ and $\mu_{\w}(F)\geq 0$ by assumption, we have $\mu_{\w}(F)\frac{k'-l'}{r'} \leq 0$; as $l' <r'$, then 
$-\frac{\delta}{r-k}(1 - \frac{l'}{r'})< 0$.
Note that if $\frac{k}{r}-\frac{k'}{r'} \leq 0$, then $\frac{\epsilon}{r-k}(\frac{k}{r}-\frac{k'}{r'}) \leq 0$ for any $\epsilon >0$.
If $\frac{k}{r}-\frac{k'}{r'} > 0$, then $\frac{k'}{r'}$ can assume finitely many values, hence we can find $\epsilon >0$ small enough in order to obtain
$$ \mu_{\wa}(E',V') < \mu_{\wa}(E,V).$$
{\bf (b2)}: $F' = F$.  Let $\wrank (G') = l' \leq k$ and $\wdeg(G') = m'$. Then $\wrank(E') = r-k+l'$. 
As in case $(b1)$ we have $\wdeg(G') \leq 0$,   so we we  have:
$$ \wdeg(E') = \wdeg(G') + \wdeg(F) = \wdeg(F) + m',$$
which implies 
\begin{equation}
\label{EQ:mueprimo}
\mu_{\underline w}(E') =  \mu_{\underline w}(F)\frac{r-k}{r-k+l'} + \frac{m'}{r-k+l'}.
\end{equation}
Let   $k' = \dim V'$. As in case (b1) we have $k' \leq l'$. So we have a  proper coherent subsystem  of $(E,V)$. 
By using Equation \eqref{EQ:mueprimo} we have
$$\mu_{\wa}(E',V') - \mu_{\wa}(E,V) = \mu_{\underline w}(E')  - \mu_{\underline w}(E) + \alpha \left(\frac{k'}{r-k+l'}- \frac{k}{r}\right) \leq $$
$$\leq \mu_{\underline w}(F) \frac{r-k}{r-k+l'} - \mu_{\underline w}(F) \frac{r-k}{r} + \frac{d-\epsilon}{r-k} \left(\frac{k'}{r-k+l'}- \frac{k}{r}\right) + \frac{m'}{r-k+l'} \leq$$
$$  \leq \mu_{\underline w}(F) \frac{k'-l'}{r-k+l'} + \frac{\epsilon}{r-k}
\left(\frac{k}{r} - \frac{k'}{r-k-l'}\right) + \frac{m'}{r-k+l'}.$$
Note that $\mu_{\underline w}(F) \frac{k'-l'}{r-k+l'} +\frac{m'}{r-k+l'} \leq 0$ and it is zero if and only if $k'= l'$ and $m'= 0$. Assume that we are in this case. Then, the injection $V' \otimes \OO_C \to G'$ is an isomorphism $G' \simeq V' \otimes \OO_C$.  Let $V = V' \oplus V''$,  then    the quotient ${E \over E'} \simeq V''  \otimes \OO_C$, hence $E \simeq  E' \oplus (V'' \otimes \OO_C)$ and then Extension \eqref{ext1} would splits. But  this is impossible since the vectors $\underline{e}_1,\dots,\underline{e}_k$ are linearly independent in $\Ext^1(F,\OO_C)$.
\hfill\par
We can conclude, as in the case (b1), that $\mu_{\underline w}(F) \frac{k'-l'}{r-k+l'} +\frac{m'}{r-k+l'} < 0$, so for $\epsilon$ small enough we have:
$$\mu_{\wa}(E',V') < \mu_{\wa}(E,V)$$
so $(E,V)$ is $\wa$-stable as claimed.
\vspace{2mm}

The claim about the fiber over points of $\mathcal{U}_{(C,\w)}^s(r-k,d)$ follows by Proposition \ref{PROP:dualextension}(c).

\end{proof}

\section{On the moduli space $\GwL(r\cdot \underline{1},d,k)$} 
\label{SEC:4}

Let $C$  be a nodal curve with  $\gamma$ smooth irreducible components $C_i$  of genus $g_i \geq 2$ and $\delta $ nodes. Let $\w$ be a good polarization on $C$. In this section we will consider coherent systems  $(E,V)$ of type $(r,d,k)$ with multirank $r\cdot \underline{1}$, $k<r$ and $d>0$. This case is interesting as it includes coherent systems with $E$ locally free. Coherent systems with such features which are $\wa$-stable are parametrized by the subscheme $\Gwa(r\cdot \underline{1},d,k)$ of the moduli space  $\Gwa(r,d,k)$.
Hence, we can consider the map $\eta$ defined in Theorem \ref{THM:BGNext} and restrict it to the subscheme $\GwL(r\cdot \underline{1},d,k)$,   we obtain a   morphism

\begin{equation}
\psi:\GwL(r\cdot \underline{1},d,k)\to \U_{(C,\w)}((r-k)\cdot \underline{1},d).
\end{equation}

The moduli spaces $\U_{(C,\w)}(s\cdot \underline{1},d)$ have been described by Teixidor i Bigas (see \cite{Tei91}  and \cite{Tei95}). We briefly recall the most relevant results. First of all,  for any integers  $s \geq 1$ and $d$, the moduli space $\U_{(C,\w)}(s\cdot \underline{1},d)$ is never irreducible but it is  connected and each irreducible component has dimension  $1+s^2(p_a(C)-1)$. The generic element of each irreducible component is the isomorphism class of a locally free sheaf $F$, which is $\w$-stable and  whose restrictions to $C_i$ are stable too. 
Each component is identified by a $\gamma$-uple $(d_1,\dots,d_{\gamma})$, where $d_i$ is the degree of the restriction to $C_i$ of the generic element and $\sum_{i=1}^{\gamma}d_i=d$. We will denote  by $X_{d_1,\dots,d_\gamma}$ the component  of $\U_{(C,\w)}(s\cdot \underline{1},d)$ corresponding to  $(d_1,\dots,d_{\gamma})$. The intersection of two such components consists of sheaves which are not locally free.  Finally, for $\w$ general, the number of irreducible components is $h\cdot s^{\gamma-1}$ where $h$ is the number of spanning tree in the dual graph of $C$.
\vspace{2mm}

The main result of this section is the following:
 
\begin{theorem}
\label{THM:main}
Let $(C,\w)$  be a polarized nodal curve with $\w$ good. Let $0 < k < r$  and $d >0$  integers. Then the following hold:
\begin{enumerate}[(a)]
\item the moduli space $\GwL(r\cdot \underline{1},d,k)$ is non empty if and only if  $k \leq  \frac{d + r(p_a(C)-1)}{p_a(C)}$;
\item for any irreducible component $X_{d_1,\dots,d_{\gamma}} \subset \U_{(C,\w)}((r-k)\cdot \underline{1},d)$ we have an irreducible component $Y_{d_1,\dots,d_{\gamma}} \subset \GwL(r\cdot \underline{1},d,k)$ which is birational to a Grassmanian fibration over $X_{d_1,\dots,d_{\gamma}}$;
\item any component $Y_{d_1,\dots,d_{\gamma}}$ has dimension $\beta(r,d,k)$ and the generic  element is a coherent system $(E,V)$ with  $E$ locally free;
\item the above components are the only ones
which contains coherent systems $(E,V)$ with $E$ locally free. 
\end{enumerate}
\end{theorem}

\begin{proof}
For simplicity, we set $\cU=\U_{(C,\w)}((r-k)\cdot \underline{1},d)$ and $\cU^s=\U_{(C,\w)}^s((r-k)\cdot \underline{1},d)$.
\hfill\par

Let $F \in \cU^s$, 
by Theorem \ref{THM:cs},  we have that $\psi^{-1} (F) \simeq  \Gr(k,H^0(F \otimes \omega_C)^*)$. 
Hence 
$$\dim \psi^{-1}(F) = k(h^0(F \otimes \omega_C)-k).$$
By Serre duality  and Lemma \ref{LEM:h0dual} we have
$h^1(F \otimes \omega_C) = h^0(F^*)= 0$.
 So 
 $$h^0(F \otimes \omega_C) = \chi(F \otimes \omega_C)= \wdeg(F \otimes {\omega}_C) + (r-k)(1-p_a(C)).$$ 
Since $\omega_C \otimes \cO_{C_i}= \omega_{C_i}(\Delta_i)$,  then by Lemma \ref{LEM:wdegree} we obtain
$$\wdeg(F \otimes {\omega}_C)= \Delta_{\w}(F) + \sum_{i=1}^{\gamma}[d_i + (r-k)(2g_i-2 + \delta_i)]=
d + 2(r-k)(p_a(C) -1)$$
so $h^0(F \otimes \omega_C) = d + (r-k)(p_a(C) -1)$ and
$$\dim \ \psi^{-1}(F) = k [d +(r-k)(p_a(C) -1) -k].$$

In particular, this proves $(a)$: the moduli space $\GwL(r\cdot \underline{1},d,k)$ is non empty if and only if 
$$k \leq  \frac{d + r(p_a(C)-1)}{p_a(C)}.$$
In this case, since $\cU^s$ is an open dense subset of $\cU$, we can conclude that $\psi$ is dominant.
\hfill\par
Let  $X_{d_1,\dots,d_{\gamma}} \subset 
\cU$ be the irreducible component corresponding to $(d_1,\dots,d_{\gamma})$.  We denote by 
$X^s_{d_1,\dots,d_{\gamma}}$ its open subset  corresponding to $\w$-stable sheaves.
Then $\psi^{-1} (X^s_{d_1,\dots,d_{\gamma}})$ is an irreducible quasi-projective variety: in fact   it  is a Grassmannian fibration 
over     $X^s_{d_1,\dots,d_{\gamma}}$, with fibers
$ \Gr(k,N)$ with $N=d+(r-k)(p_a(C)-1)$. Hence we have
$$\dim(\psi^{-1} (X^s_{d_1,\dots,d_{\gamma}})) =1 + (r-k)^2(p_a(C) -1) +k(N-k).$$

Let $(E,V)$ be a generic element of $\psi^{-1} (X^s_{d_1,\dots,d_{\gamma}})$.
Then $E$ is locally free and the evaluation map $ev_V$  is injective. By Proposition \ref{PROP:Petri}, it follows that the moduli space $\GwL(r\cdot \underline{1},d,k))$ is smooth at the point $(E,V)$ with dimension 
$\beta_C(r,d,k)$.  Actually, it is immediate to check that
$$\beta_C(r,d,k)= \dim (\psi^{-1} (X^s_{d_1,\dots,d_{\gamma}})).$$
This allow us to conclude that the closure of $\psi^{-1} (X^s_{d_1,\dots,d_{\gamma}})$ in $\GwL(r\cdot \underline{1},d,k)$ is an irreducible component of $\GwL(r\cdot \underline{1},d,k)$: we denote it by $Y_{d_1,\dots, d_{\gamma}}$.
\vspace{2mm}

We claim now that the above components are the only ones containing coherent systems $(E,V)$ with $E$ locally free. 
By contradiction, assume that there exists an irreducible  component $Y \subset \GwL(r\cdot \underline{1},d,k)$, such that    $Y \not= Y_{d_1,\cdots,d_{\gamma}}$ and  $Y$ contains $(E_0,V_0)$ with $E_0$ locally free.
Note that by Proposition \ref{PROP:Petri}, $(E_0,V_0)$ is a smooth point of the moduli space (and then of $Y$). This implies $\dim(Y)=\beta(r,d,k)$.

Consider the irreducible bounded set $\cW=\{\coker(\ev_V)\,|\, (E,V)\in Y\}$. 
By assumption,  for the general $(E,V)\in Y$ we have that $\coker(\ev_V)$  is $\w$-semistable but not $\w$-stable. As $\w$-stability is an open property, we have that all $F\in \cW$ are $\w$-semistable but not $\w$-stable.
We recall that BGN extensions of $F \in \cW$ are parametrized by the Grassmannian variety $\Gr(k,H^0(F\otimes \cO_C)^*)$.
By Lemma \ref{LEM:h0dual} we have $h^1(F\otimes \omega_C)=0$ for all $F\in \cW$, so $\dim \Gr(k,H^0(F\otimes \cO_C)^*)=k(N-k)$ with $N$ as above. 

Moreover, since $\wa$-stability is an open condition, for any $F \in \cW$ there is an open subset of $\Gr(k,H^0(F\otimes \cO_C)^*)$ parametrizing $\wa$-stable coherent systems of $Y$ with  $\coker(\ev_V)= F$. This implies that $\cW$ depends on  $1+(r-k)^2(p_a(C)-1)$ parameters. We claim that this is not possible. Any $F\in \cW$ fits into an exact sequence
$$ 0 \to F \to \bigoplus_{i=1}^{\gamma} F_i \to T \to 0, $$
where $F_i$ is the restriction of $F$ modulo torsion to $C_i$ and $T$ is a torsion sheaf whose support is contained in the set of nodes (see \cite{Ses}).  A general element of $\cW$ is locally free and its restriction $F_i$ is a locally free sheaf of rank $r-k$ and degree $d_i$ on $C_i$.
Let 
$$\cW_i=\{F_i \,|\, F \in \cW, F \ \text{locally free} \},  \ i=1, \dots \gamma.$$
As $h^1(F \otimes {\omega}_C)=0$, we have $h^1(F_i \otimes {\omega}_{C_i}(\Delta_i))= 0$ for any $i$. This implies that the set ${\cW}_i$ is a bounded set, as, up to tensoring with a {\it fixed} ample line bundle, all the elements of $\cW_i$ can be seen as quotient of a fixed trivial bundle.
In particular, $\cW_i$ depends on at most $1 + (r-k)^2(g_i -1)$ parameters, by \cite[Remark 4.2]{BGN97}. 

A  general element of $\cW$  is obtained by glueing its restrictions  at the nodes by choosing an isomorphism between the fibers. By a dimensional count it turns out that   $\cW_i$ depends actually on  $1 + (r-k)^2(g_i -1)$ parameters, hence it contains all stable vector bundles of rank $r-k$ and degree $d_i$. 
By \cite{Tei95}, a general $F$ obtained in this way is actually $\w$-stable. This is impossible as we have seen that elements of $\cW$ are never $\w$-stable.
\end{proof}

\begin{remark}
It is easy to see that if $d\geq r$, then $\GwL(r\cdot \underline{1},d,k)$ is not empty whenever $0<k<r$.
\end{remark}

\begin{remark}
We stress that, a priori, there could  be other components of the moduli space $\GwL(r\cdot \underline{1},d,k)$ besides the components $Y_{d_1,\dots,d_{\gamma}}$ defined in Theorem \ref{THM:main}. These should contain coherent systems $(E,V)$ consisting of   depth one sheaves $E$  which are not locally free such that $\coker(\ev_V)$  are  not $\w$-stable.
\end{remark}

Let $(E,V) \in Y_{d_1,\dots,d_{\gamma}}$,  we can consider its restriction $(E_i,V_i)$ to the component $C_i$. We wonder if  $(E_i,V_i)$ is $\alpha$-semistable for some $\alpha$. This does never happen when $deg(E_i)<0$ by \cite[Cor 3.2]{BDGW}.

\begin{remark}
Notice that, in general, taking restriction does not preserve stability properties. For example, $\w$-stable locally free  sheaves can have restrictions which are not even semistable (see \cite{BFKer}).  
\end{remark}

We denote by $\mathcal{G}_{C_i,L}(r,d_i,k)$  the terminal moduli space for $\alpha$-stable coherent systems of type $(r,d_i,k)$ on the curve $C_i$.  
Then we have the following: 

\begin{corollary}
\label{COR:restriction}
 Let  $Y_{d_1,\dots,d_{\gamma}}$ be an irreducible component  of $\GwL(r\cdot \underline{1},d,k)$ defined in  Theorem \ref{THM:main}.
Assume moreover that $d_i >0$ and $k \leq d_i + (r-k)(g_i -1)$. Then for a general coherent system $(E,V) \in Y_{d_1,\dots,d_{\gamma}}$  the restriction $(E_i,V_i)$  is an element of $\mathcal{G}_{C_i,L}(r,d_i,k)$.

\end{corollary}
\begin{proof}
Let $F \in X^s_{d_1,\dots,d_{\gamma}}$ be a general element, it is locally free and $F_i$  is stable with $deg(F_i)= d_i >0$ by assumption. Then, by Proposition \ref{PROP:restriction} it follows that  a general element $(E,V) \in \psi^{-1}(F)$ is  a BGN extension  of type $(r,d,k)$ whose restriction  $(E_i,V_i)$ to $C_i$ is a BGN extension of type $(r,d_i,k)$. By \cite{BG02}, the coherent system  $(E_i,V_i)$  is $\alpha$-stable for any $\alpha$ big enough. Hence $(E_i,V_i) \in \mathcal{G}_{C_i,L}(r,d_i,k)$.
\end{proof}


\begin{bibdiv}
\begin{biblist}


\bib{A79}{article}{
author={Altman, A.S.},
author={Kleiman, S.L.},
title={Bertini theorems for hypersurface sections containing a subscheme},
Journal={Comm.Algebra},
Volume={8},
year={1979},
pages={775-790.}
}

\bib{ACG}{book}{
   author={Arbarello, E.},
   author={Cornalba, M.},
   author={Griffiths, P. A.},
   title={Geometry of algebraic curves. Volume II},
   series={Grundlehren der Mathematischen Wissenschaften [Fundamental
   Principles of Mathematical Sciences]},
   volume={268},
   note={With a contribution by Joseph Daniel Harris},
   publisher={Springer, Heidelberg},
   date={2011},
   pages={xxx+963},
   doi={10.1007/978-3-540-69392-5},
}


\bib{Ber}{article}{
   author={Bertram, A.},
   title={Stable pairs and stable parabolic pairs},
   journal={J. Algebraic Geom.},
   volume={3},
   date={1994},
   number={4},
   pages={703--724},
   issn={1056-3911},
}

\bib{Bho07}{article}{
   author={Bhosle, U. N.},
   title={Brill-Noether theory on nodal curves},
   journal={Internat. J. Math.},
   volume={18},
   date={2007},
   number={10},
   pages={1133--1150},
   issn={0129-167X},
   doi={10.1142/S0129167X07004461},
}

\bib{Bho}{article}{
    author={Bhosle, U. N.},
    title={Coherent systems on a nodal curve},
    conference={title={Moduli spaces and vector bundles}},
    book={
      series={London Math. Soc. Lecture Note Ser.},
      volume={359},
      publisher={Cambridge Univ. Press, Cambridge},
   },
   date={2009},
   pages={437--455}
}
   


\bib{Bra}{article}{
   author={Bradlow, S. B.},
   title={Coherent systems: a brief survey},
   note={With an appendix by H. Lange},
   conference={
      title={Moduli spaces and vector bundles},
   },
   book={
      series={London Math. Soc. Lecture Note Ser.},
      volume={359},
      publisher={Cambridge Univ. Press, Cambridge},
   },
   date={2009},
   pages={229--264}
}

\bib{BD91}{article}{
   author={Bradlow, S. B.},
   author={Daskalopoulos, G. D.},
   title={Moduli of stable pairs for holomorphic bundles over Riemann
   surfaces},
   journal={Internat. J. Math.},
   volume={2},
   date={1991},
   number={5},
   pages={477--513},
   issn={0129-167X},
   review={\MR{1124279}},
   doi={10.1142/S0129167X91000272},
}

\bib{BDGW}{article}{
   author={Bradlow, Steven},
   author={Daskalopoulos, Georgios D.},
   author={Garc\'{\i}a-Prada, Oscar},
   author={Wentworth, Richard},
   title={Stable augmented bundles over Riemann surfaces},
   conference={
      title={Vector bundles in algebraic geometry},
      address={Durham},
      date={1993},
   },
   book={
      series={London Math. Soc. Lecture Note Ser.},
      volume={208},
      publisher={Cambridge Univ. Press, Cambridge},
   },
   date={1995},
   pages={15--67},
}

\bib{BGMN}{article}{
   author={Bradlow, S. B.},
   author={Garc\'{\i}a-Prada, O.},
   author={Mu\~{n}oz, V.},
   author={Newstead, P. E.},
   title={Coherent systems and Brill-Noether theory},
   journal={Internat. J. Math.},
   volume={14},
   date={2003},
   number={7},
   pages={683--733},
   issn={0129-167X},
   doi={10.1142/S0129167X03002009},
}

\bib{BG02}{article}{
   author={Bradlow, S. B.},
   author={Garc\'{\i}a-Prada, O.},
   title={An application of coherent systems to a Brill-Noether problem},
   journal={J. Reine Angew. Math.},
   volume={551},
   date={2002},
   pages={123--143},
   issn={0075-4102},
   doi={10.1515/crll.2002.079},
}

\bib{BGN97}{article}{
   author={Brambila-Paz, L.},
   author={Grzegorczyk, I.},
   author={Newstead, P. E.},
   title={Geography of Brill-Noether loci for small slopes},
   journal={J. Algebraic Geom.},
   volume={6},
   date={1997},
   number={4},
   pages={645--669},
   issn={1056-3911},
}

\bib{BB12}{article}{
   author={Bolognesi, M.},
   author={Brivio, S.},
   title={Coherent systems and modular subvarieties of $\mathcal{SU}_C(r)$},
   journal={Internat. J. Math.},
   volume={23},
   date={2012},
   number={4},
   pages={1250037, 23},
   issn={0129-167X},
   doi={10.1142/S0129167X12500371}
}


\bib{BFVec}{article}{
  author={Brivio, S.},
  author={Favale, F. F.},
  title={On vector bundle over reducible curves with a node},
  date={2019},
  note={To appear in {\it Advances in Geometry}},
  doi={10.1515/advgeom-2020-0010},
}

\bib{BFKer}{article}{
  author={Brivio, S.},
  author={Favale, F. F.},
  title={On Kernel Bundle over reducible curves with a node},
  date={2020},
  journal={\it International Journal of Mathematics},
  volume={31},
  date={2020},
  number={7},
  doi={10.1142/S0129167X20500548},
}

\bib{BFCoh}{article}{
  author={Brivio, S.},
  author={Favale, F. F.},
  title={Coherent systems on curves of compact type},
  date={2020},
  journal={\it Journal of Geometry and Physics},
  volume={158},
  date={2020},
  number={ },
  doi={10.1016/j.geomphys.2020.103850},
}

\bib{BFPol}{article}{
  author={Brivio, S.},
  author={Favale, F. F.},
  title={Nodal curves and polarization with good properties},
  date={2020},
  journal={\it Preprint arXiv:2008.00753v1},
}




\bib{KN}{article}{
   author={King, A. D.},
   author={Newstead, P. E.},
   title={Moduli of Brill-Noether pairs on algebraic curves},
   journal={Internat. J. Math.},
   volume={6},
   date={1995},
   number={5},
   pages={733--748},
   issn={0129-167X},
   doi={10.1142/S0129167X95000316},
}

\bib{LP}{book}{
   author={Le Potier, J.},
   title={Lectures on vector bundles},
   series={Cambridge Studies in Advanced Mathematics},
   volume={54},
   note={Translated by A. Maciocia},
   publisher={Cambridge University Press, Cambridge},
   date={1997},
   pages={viii+251},
   isbn={0-521-48182-1},
}


\bib{New11}{article}{
   author={Newstead, P. E.},
   title={Existence of $\alpha$-stable coherent systems on algebraic curves},
   conference={
      title={Grassmannians, moduli spaces and vector bundles},
   },
   book={
      series={Clay Math. Proc.},
      volume={14},
      publisher={Amer. Math. Soc., Providence, RI},
   },
   date={2011},
   pages={121--139},
}

\bib{Ses}{book}{
   author={Seshadri, C. S.},
   title={Fibr\'{e}s vectoriels sur les courbes alg\'{e}briques},
   language={French},
   series={Ast\'{e}risque},
   volume={96},
   note={Notes written by J.-M. Drezet from a course at the \'{E}cole Normale
   Sup\'{e}rieure, June 1980},
   publisher={Soci\'{e}t\'{e} Math\'{e}matique de France, Paris},
   date={1982},
   pages={209},
}
 
\bib{Tei91}{article}{
   author={Teixidor i Bigas, M},
   title={Moduli spaces of (semi)stable vector bundles on tree-like curves},
   journal={Math. Ann.},
   volume={290},
   date={1991},
   number={2},
   pages={341--348},
   issn={0025-5831},
   doi={10.1007/BF01459249},
} 
    
\bib{Tei95}{article}{
   author={Teixidor i Bigas, M.},
   title={Moduli spaces of vector bundles on reducible curves},
   journal={Amer. J. Math.},
   volume={117},
   date={1995},
   number={1},
   pages={125--139},
   issn={0002-9327},
   doi={10.2307/2375038},
}    
    
\end{biblist}
\end{bibdiv}
\end{document}